\documentclass[12pt]{amsart}
\pdfoutput=1
\usepackage{amsmath}
\usepackage{amssymb}
\usepackage{amsthm}
\usepackage{mathrsfs}
\usepackage{fullpage}
\usepackage{graphicx}
\usepackage{color}
\usepackage[colorlinks]{hyperref}
\usepackage{multirow}
\usepackage{tikz}
\usepackage{lscape}
\usepackage{multicol}

\newcommand{\CC}{\mathbb{C}}
\newcommand{\QQ}{\mathbb{Q}}
\newcommand{\fp}{\mathbb{F}_p}
\newcommand{\fpbar}{\overline{\mathbb{F}}_p}

\newcommand{\gq}{\operatorname{Gal}(\overline{\QQ}/\QQ)}
\newcommand{\GL}{\operatorname{GL}}

\newcommand{\longto}{\longrightarrow}
\newcommand{\NN}{\mathbb{N}}

\newcommand{\SL}{\operatorname{SL}}

\newcommand{\tht}{\vartheta}
\newcommand{\ZZ}{\mathbb{Z}}

\newtheorem{thm}{Theorem}
\newtheorem{cor}[thm]{Corollary}
\newtheorem{prop}[thm]{Proposition}
\newtheorem{lem}[thm]{Lemma}

\theoremstyle{definition}
\newtheorem*{ack}{Acknowledgements}
\newtheorem*{ex}{Example}
\theoremstyle{remark}
\newtheorem*{rem}{Remark}

\begin{document}
\title{Computing level one Hecke eigensystems (mod $p$)}
\author[C. Citro]{Craig Citro}
\author[A. Ghitza]{Alexandru Ghitza}
\thanks{The first author was partly supported by NSF grant DMS-0713225.  Some
of the computations described in this paper were performed on W.~Stein's 
machines {\tt mod} and {\tt geom}, supported by NSF grant DMS-0821725.}
\address{
  Department of Mathematics\\
  University of Washington\\
  Box 354350\\
  Seattle, WA 98195-4350\\
  USA
}
\email{craigcitro@gmail.com}
\address{
  Department of Mathematics and Statistics\\
  The University of Melbourne\\
  Parkville, VIC, 3010\\
  Australia
}
\email{aghitza@alum.mit.edu}

\begin{abstract}
  We describe an algorithm for enumerating the set of level $1$ systems
  of Hecke eigenvalues arising from modular forms (mod $p$).
\end{abstract}

\maketitle

\section{Introduction}

One of the cornerstone results of the modern arithmetic theory of modular
forms associates to every level $1$ Hecke eigensystem mod $p$ a unique odd
semisimple $2$-dimensional Galois representation (mod $p$) unramified
outside $p$.  This follows from the corresponding results of Deligne (and
Serre, and Eichler-Shimura) for eigenforms over $\ZZ$; a more direct
approach that avoids using the full machinery of Deligne's characteristic
zero theorem can be found in Proposition~11.1 of~\cite{Gross1990}.

Serre's conjecture (now a theorem of Khare-Wintenberger) says that all
Galois representations described above arise from level $1$ eigensystems.
In Section 8 of~\cite{Khare2007}, Khare recalls the well-known fact that
the set of level $1$ eigensystems (mod $p$) is finite of cardinality
$O(p^3)$ as $p\to\infty$, and he outlines an argument due to Serre showing
that this cardinality is $o(p^2)$ as $p\to\infty$.  Khare adds: ``\emph{It
will be of interest to get quantitative refinements of this},'' and guesses
that the cardinality is in fact asymptotic to $p^3/48$ as $p\to\infty$.  In
his PhD thesis, Centeleghe studies this question and proposes a precise
conjecture for the asymptotic behavior of the number of representations of
fixed conductor $N$ (see Conjecture~4.1.1 in~\cite{Centeleghe2009}).

The present paper describes an efficient algorithm for enumerating the set
of level $1$ eigensystems (mod $p$), and hence also the set of odd
semisimple $2$-dimensional Galois representations (mod $p$) unramified
outside of $p$.  The theoretical framework underlying our approach is based on
Tate's theory of theta cycles.  We use two alternative computational methods: 
the Victor Miller basis for modular forms of level $1$, and modular symbols
over finite fields.

In his recent preprint~\cite{Centeleghe2011}, Centeleghe attacks the
problem of counting the number of irreducible Galois representations by an
ingenious approach that requires computing with a single Hecke operator for
each prime $p$.  Unfortunately, this method only gives a lower bound on the
number of representations.  It is worth noting, however, that this lower
bound is generally very close to the known upper bound, and in many cases
(164 of the 299 cases considered in~\cite{Centeleghe2011}) allows one to
deduce the exact number.  See Section~\ref{sect:computation} for more on
the relationship between Centeleghe's work and ours.

\begin{ack}
  The authors wish to thank: C.~Khare for motivating and encouraging our work;
  W.~Stein for suggesting the use of the Victor Miller basis, useful discussions
  about decompositions into Hecke eigenspaces; K.~Buzzard and T.~Centeleghe for
  suggestions and corrections.
\end{ack}

\section{Review of modular forms mod $p$}

We recall the definition of modular forms mod $p$ of level $1$ and of
their Hecke operators.

Let $M_k(\CC)$ denote the complex vector space of holomorphic modular
forms of weight $k$ and level $1$.  There is a $\CC$-linear map that
associates to each modular form its $q$-expansion at the (only) cusp
$\infty$:
\begin{equation*}
  Q\colon M_k(\CC)\longto\CC[[q]],\quad 
  f\longmapsto f(q)=\sum_{n=0}^\infty a_nq^n.
\end{equation*}
By the $q$-expansion principle (Theorem~1.6.1 in~\cite{Katz1973}), 
this map is injective.

We define the $\ZZ$-module of forms with integer coefficients by
\begin{equation*}
  M_k(\ZZ)=Q^{-1}\left(\ZZ[[q]]\right)
\end{equation*}
and, for any $\ZZ$-module $R$, we define the $R$-module of forms with
$R$-coefficients by
\begin{equation*}
  M_k(R)=M_k(\ZZ)\otimes_{\ZZ} R.
\end{equation*}

We will define\footnote{Morally, the appropriate definition of 
modular forms mod $p$ is intrinsic, as global sections of line
bundles over the moduli stack of elliptic curves over $\fpbar$
(see Section~1.1 in~\cite{Katz1973}, Section~10 in~\cite{Gross1990},
or Section~2.1 in~\cite{Edixhoven1992}).
The naive definition we use is equivalent in level $1$ for $p\geq 5$,
by Theorem~1.8.2 and Remark~1.8.2.2 in~\cite{Katz1973}.}
the space of modular forms mod $p$ of level $1$ and
weight $k$ to be $M_k=M_k(\fp)$.  These are obtained by reducing
modulo $p$ the $q$-expansions of the modular forms with integer
coefficients. 

\subsection{Eisenstein series mod $p$}\label{sect:eis_mod_p}
There are two normalisations for Eisenstein series in characteristic zero.  
The first makes the coefficient of $q$ be one:
\begin{equation}\label{eq:gk}
  G_k=-\frac{B_k}{2k}+\sum_{n=1}^\infty \sigma_{k-1}(n)q^n.  
\end{equation}
The second one makes the constant coefficient be one:
\begin{equation}\label{eq:ek}
  E_k=-\frac{2k}{B_k} G_k=1-\frac{2k}{B_k}\sum_{n=1}^\infty \sigma_{k-1}(n)q^n.  
\end{equation}
We define Eisenstein series (mod $p$) by reducing the characteristic zero
Eisenstein series modulo $p$.  The first normalisation is problematic for
primes dividing the denominator of $B_k/(2k)$; by the von Staudt-Kummer
congruences (see Lemma~4 in~\cite{SwinnertonDyer1973}), this happens if and only
if $k$ is a multiple of $p-1$.  

\emph{Convention: To simplify notation, we will always
write $G_k$ to denote the Eisenstein series (mod $p$) of weight $k$, keeping in
mind that it is the reduction modulo $p$ of the $q$-expansion in~\eqref{eq:gk} 
if $k$ is not a multiple of $p-1$, and the reduction modulo $p$ of the 
$q$-expansion in~\eqref{eq:ek} if $k$ is a multiple of $p-1$.}

Since we will soon restrict our attention to forms of weight $\leq p+1$, the
latter situation will only occur for the \emph{Hasse invariant} $A$, which is
the reduction modulo $p$ of $E_{p-1}$.  The von Staudt-Kummer congruences
tell us that, apart from the constant coefficient, all coefficients of
$E_{p-1}$ are divisible by $p$, so the $q$-expansion of $A$ is simply
$A(q)=1\in\fp[[q]]$.

\subsection{Operators}

The spaces $M_k$ are equipped with a number of interesting linear
maps.  We will define them in the most economical way, by describing
their effect on $q$-expansions.  Suppose $f\in M_k$ has $q$-expansion
\begin{equation*}
  f(q)=\sum_{n=0}^\infty a_nq^n.
\end{equation*}

For every prime $\ell$, there is a Hecke operator $T_\ell\colon M_k\to M_k$ 
given by
\begin{equation*}
  (T_\ell f)(q)=\sum_{n=0}^\infty a_{n\ell} q^n + 
  \ell^{k-1}\sum_{n=0}^\infty a_n q^{n\ell}.
\end{equation*}

%
%

An important map is multiplication by the Hasse invariant $A$, defined 
in~\ref{sect:eis_mod_p}.  As we mentioned above, $A$ has $q$-expansion 
$A(q)=1$.  Multiplication by $A$ is an injective linear map
\begin{equation*}
  M_k\longto M_{k+(p-1)},\quad f\longmapsto Af.
\end{equation*}
Of course, it behaves like the identity map on the level of
$q$-expansions, and therefore commutes with the Hecke operators $T_\ell$.

If $f$ is a modular form (mod $p$), its \emph{filtration} is defined by
\begin{equation*}
  w(f)=\min\{k\in\NN\mid f=A^i g\text{ for some }g\in M_k, i\in\NN\}.
\end{equation*}

\subsection{The algebra of modular forms}
The product of a form of weight $k_1$ and a form of weight $k_2$ is a
modular form of weight $k_1+k_2$.  We take this multiplicative
structure into account by setting
\begin{equation*}
  M = \bigoplus_{k\in\ZZ} M_k.
\end{equation*}
This is a graded $\fp$-algebra of Krull dimension $2$.  The
$q$-expansion map
\begin{equation*}
  M\longto \fp[[q]],\quad f\longmapsto f(q)
\end{equation*}
is an algebra homomorphism with kernel $(A-1)M$.

\subsection{The theta operator}

There is a derivation on $M$, raising degrees by $p+1$:
\begin{equation*}
  \tht\colon M_k\longto M_{k+(p+1)},\quad f\longmapsto q\frac{d}{dq}f,
\end{equation*}
whose effect on $q$-expansions is
\begin{equation*}
  (\tht f)(q)=\sum_{n=0}^\infty n a_n q^n.
\end{equation*}

Katz gave a geometric construction of this operator and described some of its
properties in~\cite{Katz1977}.  Of these, we will need

\begin{prop}[Theorem (2) and Corollary (5) of~\cite{Katz1977}]
\label{prop:ker_theta}
\
\begin{enumerate}
\item If $f\in M_k$ has filtration $k$ and $p$ does not divide $k$, then 
$\tht f$ has filtration $k+p+1$.  
\item If $f\in M_k$ has $\tht(f)=0$, then $f$ has a unique expression of the 
form
\begin{equation*}
f = A^r g^p,
\end{equation*}
where $0\leq r\leq p-1$, $r+k\equiv 0\pmod{p}$, $g\in M_\ell$ and 
$p\ell+r(p-1)=k$.
\end{enumerate}
\end{prop}

We use this to find out whether an \emph{eigenform} can be in the kernel of 
$\tht$:

\begin{prop}\label{prop:ker_theta_eigen}
  If $f$ is a Hecke eigenform and $\tht^i(f)=0$ for some $i$, then $f$
  is a scalar multiple of some power of the Hasse invariant $A$.
\end{prop}
\begin{proof}
  We start by proving the case $i=1$.

  By Proposition~\ref{prop:ker_theta}, the $q$-expansion of $f\in\ker\tht$ is
  of the form
  \begin{equation*}
    f(q)=a_0+a_pq^p+a_{2p}q^{2p}+\ldots
  \end{equation*}
  Since $f$ is an eigenvector for $T_p$ (say with eigenvalue $a(p)$), we have
  \begin{equation*}
    a(p)a_0+a(p)a_pq^p+\ldots = a(p)f(q) = (T_p f)(q) = a_0+a_pq+\ldots
  \end{equation*}
  We conclude that $a_p=0$, but then $a_{np}=0$ for all $n\geq 1$.  So the
  $q$-expansion of $f$ is actually constant $f(q)=a_0$.  We normalize $f$
  so that $f(q)=1$.  Then $A-f$ is in the kernel of the $q$-expansion 
  homomorphism, so
  \begin{equation*}
    A-f=(A-1)h\quad\text{for some }h=\sum_{j=0}^N h_j\in M,
  \end{equation*}
  where $h_j$ is homogeneous of degree $j$.
  
  We distinguish three possibilities:
  \begin{enumerate}
    \item The weight of $f$ is $p-1$.  Then $f$ and $A$ are both in $M_{p-1}$
      and have the same $q$-expansion, so by the $q$-expansion principle
      $f=A$.
    \item The weight of $f$ is $<p-1$.  Then comparing the highest degree
      terms in $A-f=Ah-h$ we see that $A=Ah_N$, which means that $h=1$ and
      $f=1$.
    \item The weight of $f$ is $>p-1$.  By looking at the highest degree 
      terms in $-f+A=Ah-h$ we get $f=-Ah_N$.  Note that $0=\tht(f)=\tht(h_N)$
      and $h_N$ is a Hecke eigenform with weight strictly less than the
      weight of $f$.  We repeat the whole argument with $f$ replaced by $h_N$,
      until we fall in one of the cases (a) or (b), and we are done since
      each step peels off a factor of $-A$.
  \end{enumerate}
  
  To finish the proof, we need to consider the case $i>1$.  So suppose
  $\tht^i(f)=0$, and let $g=\tht^{i-1}(f)$.  Suppose $g\neq 0$, then $g$ is a 
  Hecke eigenform satisfying $\tht(g)=0$, so by the case $i=1$ proved above, 
  we know that $g=cA^n$ for some $c, n$.  However, since $i>1$, $g$ is in the 
  image of $\tht$, hence $g=cA^n$ is a cusp form, which implies that $g=0$.
  We can therefore move all the way down to $\tht(f)=0$, from which we
  conclude by using the case $i=1$.
\end{proof}

\subsection{Hecke eigensystems}

In view of our interest in Galois representations unramified outside
$p$, we define the (away-from-$p$) Hecke algebra by
\begin{equation*}
  \mathscr{H}=\ZZ[T_\ell\mid\ell\neq p].
\end{equation*}

By a \emph{Hecke eigensystem} we will mean a ring homomorphism
\begin{equation*}
  \Phi\colon \mathscr{H} \longto \fpbar.
\end{equation*}

It is clear that the spaces $M_k$ are $\fp\mathscr{H}$-modules.  We
say that an eigensystem $\Phi$ occurs in $M_k$ if there exists a
nonzero $f\in M_k$ such that
\begin{equation*}
  Tf = \Phi(T)f\quad\text{for all }T\in\mathscr{H}.
\end{equation*}

If $\Phi$ is an eigensystem, we define the (first) \emph{twist} of $\Phi$ by
\begin{equation*}
  \Phi[1]\colon \mathscr{H}\longto\fpbar,\quad 
  T_\ell\longmapsto \ell \Phi(T_\ell).
\end{equation*}
It is clear that this operation can be repeated (at most) $p-1$ times
before getting back to $\Phi$.  The resulting eigensystems are called
the \emph{twists} of $\Phi$.  The twisting operation has a modular
interpretation: if $\Phi$ is the eigensystem of $f\in M_k$ and $\tht f\neq 0$,
then $\Phi[1]$ is the eigensystem of $\tht f\in M_{k+p+1}$.

We will say that two eigensystems $\Phi$ and $\Psi$ are \emph{equivalent}
(write $\Phi\sim\Psi$) if $\Phi$ is a twist of $\Psi$, i.e. if there exists
$i$ such that $\Phi=\Psi[i]$.

One of the crucial results for our computational work is due to Ash
and Stevens (Theorems 3.4 and 3.5 in~\cite{Ash1986}):
\begin{thm}[see Theorem 3.4 in~\cite{Edixhoven1992}]
  Every modular eigensystem has a twist that occurs in weight $\leq
  p+1$. 
\end{thm}

This indicates that, instead of having to work with spaces of arbitrary
weight, it suffices to restrict to weight $\leq p+1$ and take twists.

\subsection{The Sturm-Murty bound}

We need to be able to decide whether two eigensystems are equal by comparing
only finitely many of the eigenvalues.  The following result (due to Sturm and 
revisited by Murty) solves this problem in the case of two eigenforms of the 
same weight:

\begin{thm}[special case of Theorem~1 in~\cite{Murty1996}]
  Let $f$ and $g$ be holomorphic modular forms of weight $k$ and level $1$, 
  with Fourier coefficients $a_f(n)$ and $a_g(n)$.  Let $\beta(k)=k/12$ and 
  suppose that
  \begin{equation*}
    a_f(n) = a_g(n)\quad\text{for all }n\leq \beta(k).
  \end{equation*}
  Then $f=g$.  
\end{thm}

The proof works in any characteristic; via the relation between Fourier 
coefficients and Hecke operators we arrive at the form in which we will use 
the result:

\begin{prop}
  Let $\Phi$ and $\Psi$ be eigensystems occurring in the same weight $k$ and 
  suppose that
  \begin{equation*}
    \Phi(\ell) = \Psi(\ell)\quad\text{for all primes }\ell\leq \beta(k).
  \end{equation*}
  Then $\Phi=\Psi$.
\end{prop}

\section{Some consequences of the theory of theta cycles}

Let $f$ be a modular form such that $\tht f\neq 0$.  The 
\emph{$\tht$-cycle} of $f$ is defined to be the $(p-1)$-tuple of
integers
\begin{equation*}
  \left(w(\tht f), w(\tht^2 f), \ldots, w(\tht^{p-1}f)\right).
\end{equation*}

A lot is known about the structure of theta cycles.  For low weights, we
will use the following classification given by Edixhoven:

\begin{prop}[Proposition 3.3 in~\cite{Edixhoven1992}]\label{prop:theta_cycle}
  Let $p\geq 5$ be prime.
  Let $f$ be an eigenform (mod $p$) of weight and filtration $k$, where
  $k\leq p+1$.  Let $(a_\ell)$ denote the eigenvalues of $f$.
  \begin{enumerate}
    \item If $a_p\neq 0$ ($f$ is \emph{ordinary}), then the $\tht$-cycle of 
    $f$ is given by
      \begin{center}
        \begin{tabular}{l | l}
          weight & $\tht$-cycle \\ \hline
          \multirow{2}{*}
          {$4\leq k\leq p-1$} & $(k+(p+1),\ldots,k+(p-k)(p+1),$ \\
          & $\phantom{(}k^{\prime}+(p+1),\ldots,k^{\prime}+(k-1)(p+1))$ \\
          $k=p+1$ & $(p+1+(p+1),\ldots,p+1+(p-1)(p+1))$\\
        \end{tabular}
      \end{center}
      where $k^{\prime}=p+1-k$.  See Figure~\ref{fig:theta_ord}.
    \item If $a_p=0$ ($f$ is \emph{non-ordinary}), then the $\tht$-cycle of 
    $f$ is given by
      \begin{center}
        \begin{tabular}{l | l}
          weight & $\tht$-cycle \\ \hline
          \multirow{2}{*}
          {$4\leq k\leq p-1$} & $(k+(p+1),\ldots,k+(p-k)(p+1),k^{\prime\prime},$ \\
          & $\phantom{(}k^{\prime\prime}+(p+1),\ldots,k^{\prime\prime}+(k-3)(p+1),k)$ \\
          $k=p+1$ & does not occur\\
        \end{tabular}
      \end{center}
      where $k^{\prime\prime}=p+3-k$.  See Figure~\ref{fig:theta_nonord}.
  \end{enumerate}
\end{prop}

\begin{rem}
  We have extracted from the statement of Proposition 3.3 
  in~\cite{Edixhoven1992} only the parts that are relevant to level $1$.  We 
  have also eliminated the unnecessary requirement that $f$ be a cusp form 
  (see Section~7 in~\cite{Jochnowitz1982}).
\end{rem}

\begin{figure}
\begin{center}
  \begin{tikzpicture}[scale=0.8]
    \tikzstyle{every node}=[draw, circle, fill=black, minimum size=4pt, inner sep=0pt]
    \draw (0,2.0) node (20) [label=left:\small{$k+(p+1)$}] {}
    -- (0.6,3.2) node (32) [] {}
    -- (1.2,4.4) node (44) [label=left:\small{$k+(p-k)(p+1)$}] {};
    \draw[loosely dotted] (44)
    -- (1.8,1.6) node (16) [label=right:\small{$k^\prime+(p+1)$}] {};
    \draw (16)
    -- (2.4,2.8) node (28) [] {}
    -- (3.0,4.0) node (40) [] {}
    -- (3.6,5.2) node (52) [] {}
    -- (4.2,6.4) node (64) [] {}
    -- (4.8,7.6) node (76) [] {}
    -- (5.4,8.8) node (88) [label=left:\small{$k^\prime+(k-1)(p+1)$}] {};
    \draw[loosely dotted] (88) .. controls (8,0) and (1.8,0) .. (20);

    \draw (8,2.4) node (24) [label=right:\small{$k+(p+1)$}] {}
    -- ++(0.6,1.2) node (36) [] {}
    -- ++(0.6,1.2) node (48) [] {}
    -- ++(0.6,1.2) node (60) [] {}
    -- ++(0.6,1.2) node (72) [] {}
    -- ++(0.6,1.2) node (84) [] {}
    -- ++(0.6,1.2) node (96) [] {}
    -- ++(0.6,1.2) node (108) [] {}
    -- ++(0.6,1.2) node (120) [] {}
    -- ++(0.6,1.2) node (132) [label=right:\small{$k+(p-1)(p+1)$}] {};
    \draw[loosely dotted] (132) .. controls (14,1) and (10,1) .. (24);
  \end{tikzpicture}
\end{center}
\caption{Theta cycles of ordinary forms: $4\leq k\leq p-1$ (left, 
$k^\prime=p+1-k$), $k=p+1$ (right).}
\label{fig:theta_ord}
\end{figure}
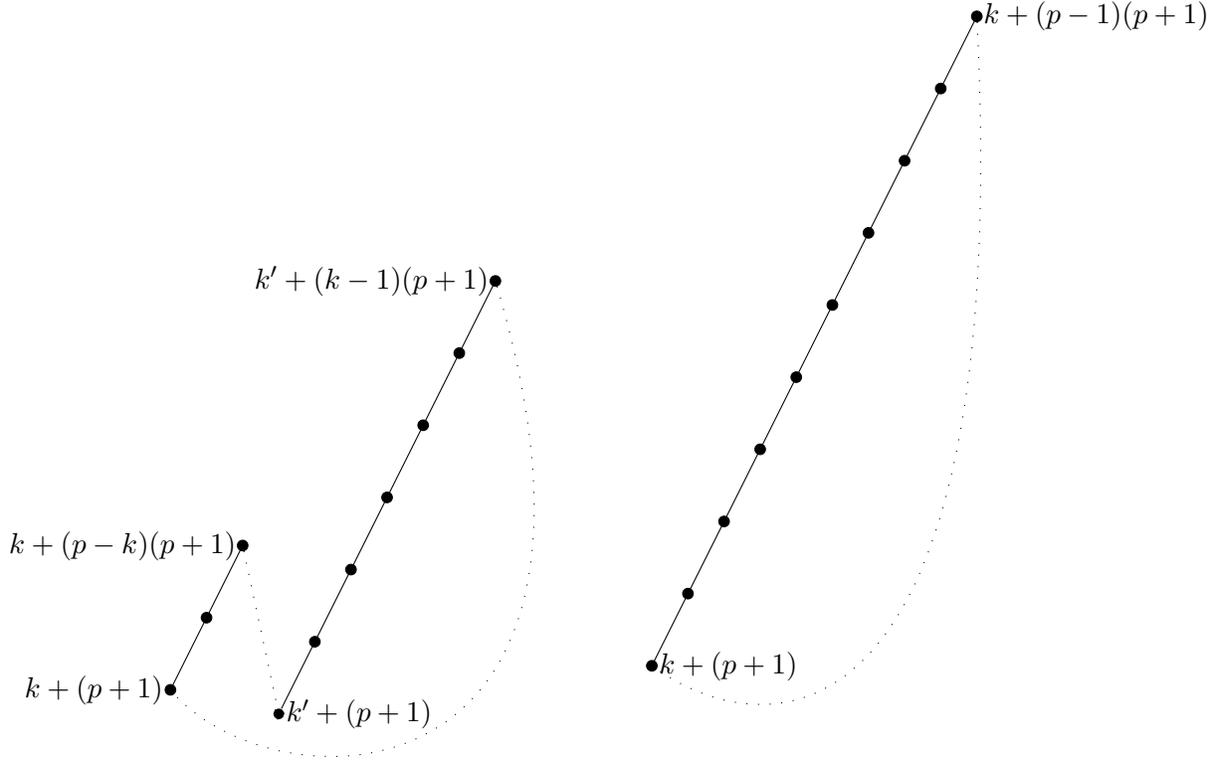

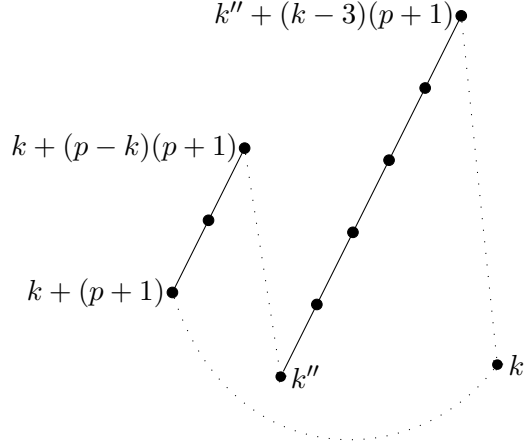
\begin{figure}
\begin{center}
  \begin{tikzpicture}[scale=0.8]
    \tikzstyle{every node}=[draw, circle, fill=black, minimum size=4pt, inner sep=0pt]
    \draw (0,2.0) node (20) [label=left:\small{$k+(p+1)$}] {}
    -- (0.6,3.2) node (32) [] {}
    -- (1.2,4.4) node (44) [label=left:\small{$k+(p-k)(p+1)$}] {};
    \draw[loosely dotted] (44)
    -- (1.8,0.6) node (6) [label=right:\small{$k^{\prime\prime}$}] {};
    \draw (6)
    -- (2.4,1.8) node (18) [] {}
    -- (3.0,3.0) node (30) [] {}
    -- (3.6,4.2) node (42) [] {}
    -- (4.2,5.4) node (54) [] {}
    -- (4.8,6.6) node (66) [label=left:\small{$k^{\prime\prime}+(k-3)(p+1)$}] {};
    \draw[loosely dotted] (66)
    -- (5.4,0.8) node (8) [label=right:\small{$k$}] {};
    \draw[loosely dotted] (8) .. controls (4,-1) and (1,-1) .. (20);
  \end{tikzpicture}
\end{center}
\caption{Theta cycle of a nonordinary form; $4\leq k\leq p-1$, 
$k^{\prime\prime}=p+3-k$.}
\label{fig:theta_nonord}
\end{figure}

\begin{lem}\label{lem:equiv_theta}
  Let $f_1$ and $f_2$ be eigenforms with $\Phi(f_1)\sim\Phi(f_2)$.
  Then the $\tht$-cycles of $f_1$ and $f_2$ are the same up to a cyclic
  permutation.
\end{lem}
\begin{proof}
  We start by reducing to the case where neither $f_1$ nor $f_2$ is in the
  kernel of $\tht$.  Suppose $f_1\in\ker(\tht)$, then by 
  Proposition~\ref{prop:ker_theta_eigen} we know that $f_1=cA^n$ for some $c, n$.
  Therefore $\Phi(f_1)=\Phi(A)=\Phi(G_{p+1})[p-2]$, so we may replace $f_1$
  by $G_{p+1}$, which is not in the kernel of $\tht$.  The same goes for
  $f_2$.

  Since the eigensystems are equivalent, there exists an integer $i$ such
  that $\Phi(f_1)=\Phi(\tht^i f_2)$.  In particular, the weight of $f_1$
  and the weight of $\tht^i f_2$ are congruent modulo $p-1$.  We have that
  $\tht(f_1)\neq 0$ and $\tht(\tht^i f_2)\neq 0$, so $\tht(f_1)$ and
  $\tht^{i+1}(f_2)$ have the same $q$-expansion, and their weights are
  congruent modulo $p-1$.  Without loss of generality, the weight of
  $\tht(f_1)$ is less than or equal to the weight of $\tht^{i+1}(f_2)$, so
  there exists $j$ such that $A^j\tht(f_1)$ has the same weight as
  $\tht^{i+1}(f_2)$.  These forms also have the same $q$-expansion, so they
  must be equal:
  \begin{equation*}
    A^j\tht f_1=\tht^{i+1} f_2.
  \end{equation*}
  But then for all $a\geq 1$ we have
  \begin{equation*}
    A^j\tht^a f_1=\tht^{i+a} f_2.
  \end{equation*}
  Since $w(Ag)=w(g)$ for all modular forms $g$, we conclude that the
  $\tht$-cycles of $f_1$ and $f_2$ are the same up to a cyclic permutation.
\end{proof}

We use Edixhoven's result to determine when two eigensystems are equivalent,
and to estimate the number of twists of a given eigensystem.

\begin{thm}\label{thm:equiv}
  Let $f_1$, respectively $f_2$ be eigenforms of weight and filtration
  $k_1$, respectively $k_2$, where
  \begin{equation*}
    1\leq k_1<k_2\leq p+1.
  \end{equation*}
  Suppose $\Phi(f_1)\sim\Phi(f_2)$, then we must be in one of the following
  two situations:
  \begin{enumerate}
    \item $a_p(f_1)\neq 0\neq a_p(f_2)$ and $k_1+k_2=p+1$;
    \item $a_p(f_1)= 0= a_p(f_2)$ and $k_1+k_2=p+3$.  
  \end{enumerate}
\end{thm}
\begin{proof}
  By Lemma~\ref{lem:equiv_theta}, the $\tht$-cycles of $f_1$ and $f_2$ 
  are the same up to a 
  cyclic permutation.  The two cases now follow by comparing the general
  shape and the low points of the cycles in Edixhoven's classification.
\end{proof}

\begin{prop}\label{prop:nr_twists}
  Let $f$ be an eigenform of weight and filtration $k$, where
  $1\leq k\leq p+1$, and let $\Phi=\Phi(f)$.  Let $n(\Phi)$ denote the
  number of distinct twists of $\Phi$.  Then $n(\Phi)=p-1$ unless
  $a_p\neq 0$ and $k=\frac{p+1}{2}$, in which case
  \begin{equation*}
    n(\Phi)\in\left\{\frac{p-1}{2},p-1\right\};
  \end{equation*}
\end{prop}
\begin{proof}
  Suppose $n(\Phi)\neq p-1$.  Then $n(\Phi)$ is a divisor of $p-1$, and the
  $\tht$-cycle of $f$ consists of copies of subcycles of length $n(\Phi)$.

  According to Edixhoven's classification, the only cycle
  that could have this shape occurs when $a_p\neq 0$, $4\leq k\leq p-1$.  
  This shape has two low points, so $n(\Phi)\geq (p-1)/2$.  Moreover, the two 
  low points must agree, i.e.
  \begin{equation*}
    k+p+1 = 2p+2-k\quad\Rightarrow\quad k=\frac{p+1}{2}.
  \end{equation*}
\end{proof}

An immediate consequence is:

\begin{cor}\label{cor:twists1}
  If $p\equiv 1\pmod{4}$ is a prime, then every Hecke eigensystem mod $p$
  has exactly $p-1$ twists.
\end{cor}

\begin{ex}
  In Section~\ref{sect:eisenstein} we prove that if $p\equiv 3\pmod{4}$, 
  $G_{(p+1)/2}$ always has $\tht$-cycle of length $(p-1)/2$.

  If $f$ is a cusp form of weight $(p+1)/2$, its $\tht$-cycle length can 
  be either $(p-1)/2$ or $p-1$.  We give an explicit example for each of these 
  two cases.  

  \begin{enumerate}
  \item The smallest example of a cusp form of weight $(p+1)/2$ with theta 
  cycle of length $(p-1)/2$ is $\Delta$ mod $23$:
  \begin{equation*}
    \Delta(q)=q + 22q^2 + 22q^3 + q^6 + q^8 + 22q^{13} + 22q^{16} + q^{23} + 
    22q^{24} + q^{25} + O(q^{26}).
  \end{equation*}

  I claim that $\tht^{12}\Delta =A^{12}\tht\Delta$, and hence the $\tht$-cycle
  of $\Delta$ has length $11$.  This alleged equality takes place in weight
  $300$, where the Sturm bound is $25$, so it suffices to check it on
  $q$-expansions up to that precision:
  \begin{eqnarray*}
    (\tht^{12} \Delta)(q) &=& q + 21q^{2} + 20q^{3} + 6q^{6} + 8q^{8} + 10q^{13} + 7q^{16} + 22q^{24} + 2q^{25} + O(q^{26}),\\
    (A^{12}\tht \Delta)(q) &=& q + 21q^{2} + 20q^{3} + 6q^{6} + 8q^{8} + 10q^{13} + 7q^{16} + 22q^{24} + 2q^{25} +  O(q^{26}).
  \end{eqnarray*}
  \item 
  The smallest example of a cusp form of weight $(p+1)/2$ with theta cycle of
  length $p-1$ occurs for $p=43$.  The space of cusp forms of weight $22$ is
  one-dimensional; denote its normalized generator by $\Delta_{22}$ (an 
  explicit expression for it is $\Delta_{22}=41 G_4^4 G_6+18 G_4 G_6^3$).  The 
  beginning of its $q$-expansion is
  \begin{equation*}
    \Delta_{22}(q) = q + 13q^{2} + 27q^{3} + 41q^{4} + 39q^{5} + O(q^6).
  \end{equation*}
  The following shows that the $\tht$-cycle length is not $21$:
  \begin{eqnarray*}
    (\tht^{22} \Delta_{22})(q) &=& q + 13q^{2} + 4q^{3} + 18q^{4} + 16q^{5} + O(q^{6}),\\
    (A^{22}\tht \Delta_{22})(q) &=& q + 3q^{2} + 12q^{3} + 3q^{4} + 11q^{5} + O(q^{6}).
  \end{eqnarray*}
  \end{enumerate}
\end{ex}

\section{Eigensystems coming from Eisenstein series}\label{sect:eisenstein}

\begin{prop}\label{prop:eisenstein}
  Let $4\leq k_1< k_2\leq p+1$ and let $\Phi_1$,
  $\Phi_2$ denote the eigensystems of the Eisenstein series
  $G_{k_1}$ and $G_{k_2}$.  Then $\Phi_1\sim\Phi_2$ if and only if 
  $k_1+k_2\equiv 2\pmod{p-1}$.  In this case, $\Phi_2=\Phi_1[p-k_1]$.  
\end{prop}
\begin{proof}
  Suppose $k_1+k_2\equiv 2\pmod{p-1}$.  On one hand
  we have
  \begin{equation*}
    \Phi_1[p-k_1](T_\ell)=\ell^{p-k_1}\left(1+\ell^{k_1-1}\right)=\ell^{p-k_1}+1.
  \end{equation*}
  On the other hand:
  \begin{equation*}
    k_1+k_2\equiv 2\pmod{p-1}\quad\Rightarrow\quad
    k_2\equiv p+1-k_1\pmod{p-1},
  \end{equation*}
  so
  \begin{equation*}
    \Phi_2(T_\ell)=1+\ell^{k_2-1}=1+\ell^{p+1-k_1-1}.
  \end{equation*}

  For the other implication, suppose
  $\Phi_2=\Phi_1[i]$ for some $i$.  This means that
  \begin{equation*}
    \ell^i + \ell^{i+k_1-1} \equiv 1 + \ell^{k_2-1}\pmod{p}
  \end{equation*}
  for all primes $\ell\neq p$.  Let $a$, $b$, $c$ be the respective
  remainders of the division by $p-1$ of $i$, $i+k_1-1$, $k_2-1$.  (In
  particular, $a, b, c < p-1$.)  Then in $\fp$ we have 
  \begin{equation}\label{eq:eis_cong}
    \alpha^a + \alpha^b = 1 + \alpha^c
    \quad\text{for all }\alpha\in\fp^\times.
  \end{equation}
  Consider the polynomial
  \begin{equation*}
    f(x)=x^a+x^b-1-x^c\in\fp[x].
  \end{equation*}
  The degree of $f$ is at most $p-2$ (or $f$ is the zero polynomial).
  If $f\neq 0$ then $f$ has at most $p-2$ roots in $\fp$.
  However, equation~(\ref{eq:eis_cong}) implies that $f$ has $p-1$
  roots in $\fp$, so we must have that $f=0$.

  We have two possibilities: (i) $a=0$ and $b=c$, which implies $i=0$
  and $k_1=k_2$, contradicting the assumption that $k_1<k_2$; (ii)
  $b=0$ and $a=c$, which implies
  \begin{equation*}
    k_1+k_2\equiv 2\pmod{p-1}\quad\text{and}\quad
    i\equiv k_2-1\equiv p+k_2-2\equiv p-k_1\pmod{p-1}.
  \end{equation*}
\end{proof}

\begin{prop}\label{prop:eisenstein_twists}
  Let $4\leq k\leq p+1$.  The Eisenstein series $G_k$ has $p-1$ twists, unless 
  $p\equiv 3\pmod{4}$ and $k=(p+1)/2$, in which case $G_k$ has $(p-1)/2$ 
  twists.   
\end{prop}
\begin{proof}
  The case $p\equiv 1\pmod{4}$ is taken care of by Corollary~\ref{cor:twists1}.

  Suppose now that $p\equiv 3\pmod{4}$.  Eisenstein series are always 
  ordinary, so $a_p\neq 0$.  If $k\neq (p+1)/2$ then $E_k$ has theta cycle of 
  length $p-1$ by Proposition~\ref{prop:nr_twists}.  In the remaining case 
  $k=(p+1)/2$, let $\Phi$ be the eigensystem of $G_k$.  We easily see that
  \begin{eqnarray*}
    \Phi(T_\ell)&=&1+\ell^{(p+1)/2-1}=1+\ell^{(p-1)/2}\\
    \Phi[(p-1)/2](T_\ell)&=&\ell^{(p-1)/2}\left(1+\ell^{(p-1)/2}\right)=
    \ell^{(p-1)/2}+1,
  \end{eqnarray*}
  so $\Phi$ has $(p-1)/2$ twists.   
\end{proof}

\begin{cor}\label{cor:eisenstein_number}
  The number of distinct eigensystems (mod $p$) coming from Eisenstein series
  is $(p-1)^2/4$.
\end{cor}
\begin{proof}
  This follows via simple arithmetic from Propositions~\ref{prop:eisenstein}
  and~\ref{prop:eisenstein_twists}.
\end{proof}

We end this section by discussing the possibility that an Eisenstein series 
and a cuspidal eigenform of small weights have equivalent eigensystems:

\begin{prop}\label{prop:eis_cusp}
Let $G_k$ be the Eisenstein series of weight $k\leq p+1$ and fix an even
integer $k^\prime\neq 14$ with $12\leq k^\prime\leq p+1$.  A cuspidal eigenform
$f$ of weight $k^\prime$ with $\Phi(G_k)\sim \Phi(f)$ exists if and only if
$k^\prime=k$ and $p$ divides the numerator of the $k$-th Bernoulli number
$B_k$.
\end{prop} 
\begin{proof} 
The argument can be extracted from
page~334 of~\cite{Serre1972}; we include it here for completeness.  

Suppose there exists a form $f$ with the given properties.  Then there is some 
integer $i$ such that $\Phi(f)=\Phi(G_k)[i]$, i.e. $\tht f=\tht^{i+1}G_k$.  
The conditions imposed on $k^\prime$ exclude the possibility of it being 
divisible by $p$, therefore the filtration of $\tht f$ is $k^\prime+p+1$.   
Similarly, the filtration of $\tht^{i+1}G_k$ is $k+(i+1)(p+1)$.  Therefore
\begin{equation*}
  k^\prime+p+1=k+(i+1)(p+1).
\end{equation*}
However, $k^\prime\leq p+1$ so $k^\prime+p+1\leq 2(p+1)$, from which we 
conclude that $i=0$, so $k^\prime=k$.  

Therefore $\tht (f-G_k)=0$.  Again since $k$ is not divisible by $p$ we get 
that $f=G_k$, in particular the constant term of $G_k$ is zero; but this 
constant term is the reduction modulo $p$ of $B_k/(2k)$, therefore $p$ must 
divide the numerator of $B_k/(2k)$.  Using one last time the condition 
$k\leq p+1$ we conclude that $p$ divides the numerator of $B_k/(2k)$ if and 
only if it divides the numerator of $B_k$.
\end{proof}

\section{Bounds on the number of eigensystems}\label{sect:upper_bound}
In this section, we derive an explicit formula for the well-known upper bound
on the number\footnote{We use
Khare's notation, which is motivated by the fact that this is the number
of continuous semisimple odd representations
\begin{equation*}
  \rho\colon \gq\longto\GL_2(\fpbar)
\end{equation*}
that are unramified outside $p$.  Note that we do not restrict our attention
to irreducible representations here, but by Corollary~\ref{cor:eisenstein_number}
the difference is known to be $(p-1)^2/4$.}
$N(2, p)$ of level $1$ Hecke eigensystems modulo $p$.

Let $N_{\text{twist}}(2, p)$ be the number of equivalence
classes up to twist of level~$1$ Hecke eigensystems modulo $p$.  
We have seen that any eigensystem has at most $p-1$
twists, so we get the inequality
\begin{equation*}
  N(2, p)\leq N_{\text{twist}}(2, p)\cdot(p-1).
\end{equation*}

We know that each eigensystem occurs, up to twist, in weights at most
$p+1$.  Therefore we can bound $N_{\text{twist}}(2, p)$ by the sum of
the dimensions of the spaces $M_k(\SL_2(\ZZ))$ for $k\leq p+1$:
\begin{equation*}
  N_{\text{twist}}(2, p)\leq \sum_{k=4}^{p+1} \dim M_k(\SL_2(\ZZ)).
\end{equation*}

We now use the classical dimension formulas (see, e.g. Corollary~1 in
Section~1.3 of~\cite{Zagier2008}):
\begin{equation*}
  \dim M_k(\SL_2(\ZZ))=\begin{cases}
    0 & \text{if } k<0 \text{ or $k$ is odd}\\
    \lfloor \frac{k}{12}\rfloor & \text{if }k\equiv 2\pmod{12}\\
    \lfloor \frac{k}{12}\rfloor + 1 & \text{otherwise}.
  \end{cases}
\end{equation*}

After a straightforward calculation, we obtain the following
expression for the sum of dimensions (write $Q$ for the quotient of
the integer division of $p+1$ by $12$):
\begin{equation*}
  \sum_{k=4}^{p+1} \dim M_k(\SL_2(\ZZ))=\begin{cases}
    3Q^2+4Q & \text{if }p\equiv 1\pmod{12}\\
    3Q^2+6Q+2 & \text{if }p\equiv 5\pmod{12}\\
    3Q^2+7Q+3 & \text{if }p\equiv 7\pmod{12}\\
    3Q^2+3Q & \text{if }p\equiv 11\pmod{12}.
  \end{cases}
\end{equation*}

It remains to multiply this value by $p-1$ in order to obtain the
desired upper bound on $N(2, p)$.  Note that this upper bound is asymptotic
to $p^3/48$ as $p\to\infty$.

\section{Detailed description of the algorithm}

\subsection*{Step 1: Get the eigensystems coming from Eisenstein series}

According to Proposition~\ref{prop:eisenstein}, the complete list of such 
eigensystems up to twist is: $G_k$ for $4\leq k\leq (p+1)/2$, together with 
$G_{p+1}$.

\subsection*{Step 2: Get the eigensystems coming from cusp forms of weight 
up to $p+1$} 

Fix a weight $12\leq k\leq p+1$.  We took two different approaches:
\begin{itemize}
\item Compute the (cuspidal) Victor Miller basis over $\fp$ of weight $k$ up to 
and including the $p$-th coefficient, then decompose the span of this basis 
into Hecke eigensystems.
\item Compute the (cuspidal) modular symbols of weight $k$ and sign $+1$ over
$\fp$, then decompose into Hecke eigenspaces.
\end{itemize}

This gives us a list of cuspidal eigenforms $f_1,\ldots,f_n$ with 
$n\leq\dim S_k$.

\subsection*{Step 3: Remove duplicates (up to twist)}
Three special circumstances can arise:
\begin{enumerate}

  \item If $p$ divides the numerator of $B_k$, then one of the $f_j$'s is the
  Eisenstein series $G_k$, by Proposition~\ref{prop:eis_cusp}.  Determine the
  relevant $f_j$ (need to compare $q$-expansion up to $\beta(k)$) and remove it
  from the list.  
  
  \item If $k>(p+1)/2$: for each $j\in\{1,\ldots,n\}$ such that 
  $a_p(f_j)\neq 0$, there could be a companion form in weight $p+1-k$.  In 
  this case the eigensystem of $f_j$ is a twist of an eigensystem that has 
  already been listed, so we should remove $f_j$ from the list.  Checking this 
  requires comparing the ordinary $f_j$'s with the ordinary forms of weight 
  $p+1-k$, up to precision $\beta(k+p+1)$. 

  Here is the justification for the comparison bound: we have $f$ of weight
  $k>(p+1)/2$ and $g$ of weight $p+1-k$.  We want to check whether the
  $q$-expansions $\tht f$ (in weight $k+p+1$) and $\tht^k g$ (in weight
  $kp+p+1$) are equal.  A priori it seems that this must be checked in weight
  $kp+p+1$, where we are verifying the equality $A^k \tht f=\tht^k g$.
  However, as Buzzard pointed out to us, we can do much better by using theta 
  cycles.  We are in the situation illustrated in Figure~\ref{fig:theta_ord}: 
  $\tht f$ is the first low point of the cycle, and $\theta g$ is the second low 
  point.  Following the cycle, we see that $\tht^k g$ is back at the first low 
  point, i.e. that $\tht^k g$ has filtration $k+p+1$.  Therefore it suffices to
  perform the comparison in weight $k+p+1$, checking $q$-expansions up to 
  $\beta(k+p+1)$.  

  \item If $k>(p+3)/2$: for each $j\in\{1,\ldots,n\}$ such that $a_p(f_j)=0$,
  there exists a non-ordinary form $g$ of weight $p+3-k$ with the same
  eigensystem up to twist, therefore $f_j$ should be removed from the list.

  As a consistency check (inexpensive since there are not many nonordinary
  forms), we could actually compare the nonordinary $f_j$'s to the nonordinary
  forms of weight $p+3-k$ and find the corresponding form there.  As we see
  from the theta cycle of $f$ in Figure~\ref{fig:theta_nonord}, we want to check
  whether the $q$-expansions of $\tht^{p+1-k}f$ and $g$ agree, in weight
  $p+3-k$.  We therefore need to compare up to precision $\beta(p+3-k)$.

\end{enumerate}

We now have the list of all eigensystems up to twist.

%

\section{Summary and discussion of results}\label{sect:computation}

The table appearing below in the appendix records, for all the primes under
$2000$, the number of distinct non-Eisenstein\footnote{We decided to exclude
the Eisenstein eigensystems from the count in order to ease comparison with
Centeleghe's results.  As Corollary~\ref{cor:eisenstein_number} indicates,
the number of Eisenstein eigensystems (mod $p$) is $(p-1)^2/4$.}
eigensystems mod $p$, the upper bound on this number, as well as any
interesting features that each prime might have.  The latter are denoted by
an E for Eisenstein-cuspidal congruences, a C for companion forms, or an
N for nonordinary forms, followed by the weights in which the corresponding
phenomenon occurs.  Note that companion forms and nonordinary forms always
show up in pairs, but only the smallest weight is recorded in the table for
each such pair. 

The first explicit examples of companion forms appear in~\cite{Gross1990},
resulting from computations done by Elkies and Atkin.  They focused on
finding primes at which the reduction of the six cuspidal eigenforms with
rational coefficients have companions.  Higher degree examples were given
by Centeleghe in his thesis~\cite{Centeleghe2009}, going up to $p=619$.  Our
results extend this range to all $p<2000$.

Similarly, we find new examples of nonordinary forms mod $p<2000$ of weight
$k\leq p+1$, extending those listed in Tables 5 and 6 of~\cite{Centeleghe2009}
and the results of Gouv\^ea in~\cite{Gouvea1997}.

It is interesting to compare our results with Centeleghe's table
in~\cite{Centeleghe2011}.  Out of the 299 lower bounds he computes, 164 are 
marked with a star, meaning that they are proved to give the
actual number of representations.  Our results indicate that a further 111
of his lower bounds coincide with the exact numbers, for a total of 275 out of
299.

Finally, we notice that the ``interesting'' phenomena described above are
quite rare, and the actual number of eigensystems deviates very little from
the explicit upper bound given in Section~\ref{sect:upper_bound}.  For instance,
among the last 20 primes in the range we computed, the relative difference
between the actual number and the upper bound is always less than $0.017\%$.
This relative ``error'' is plotted in Figure~\ref{fig:reldiff} at three
different zoom levels.  Note that the primes congruent to $1$ mod $4$,
represented by blue discs in the figure, generally tend to be closer to the
upper bound than the primes congruent to $3$ mod $4$.

\begin{figure}[h]
\centering
\includegraphics[scale=0.43]{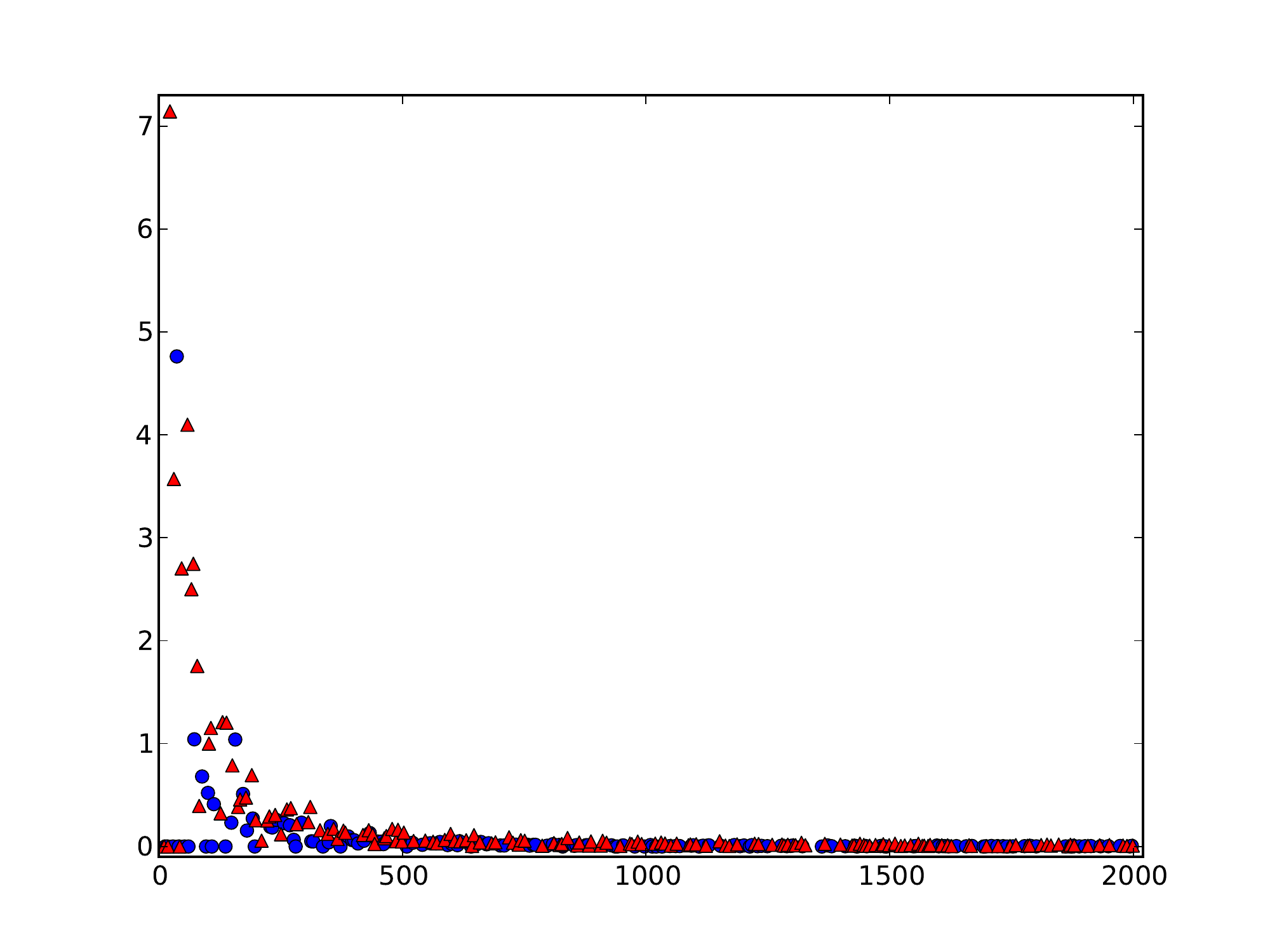}
\includegraphics[scale=0.43]{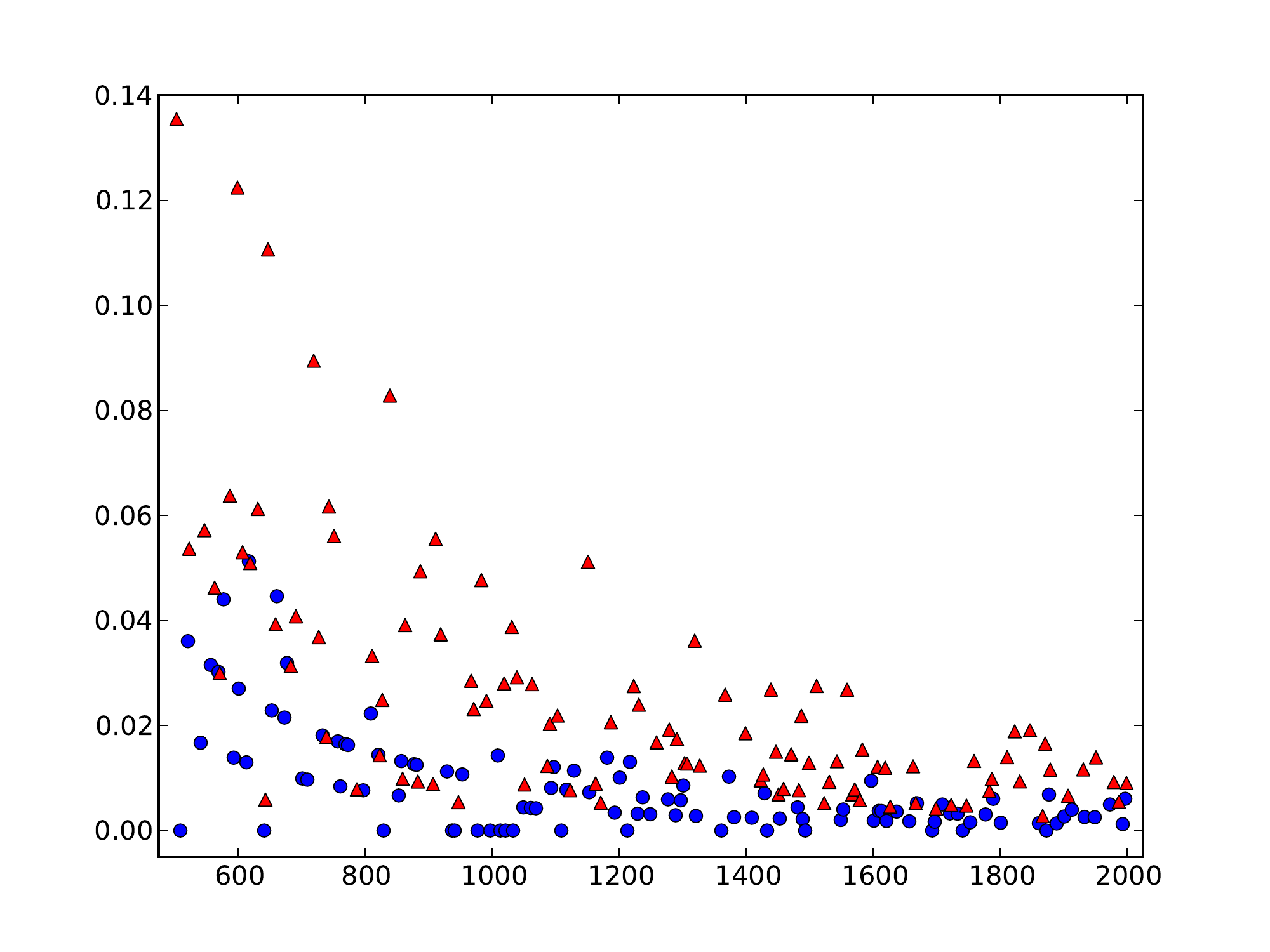}
\includegraphics[scale=0.43]{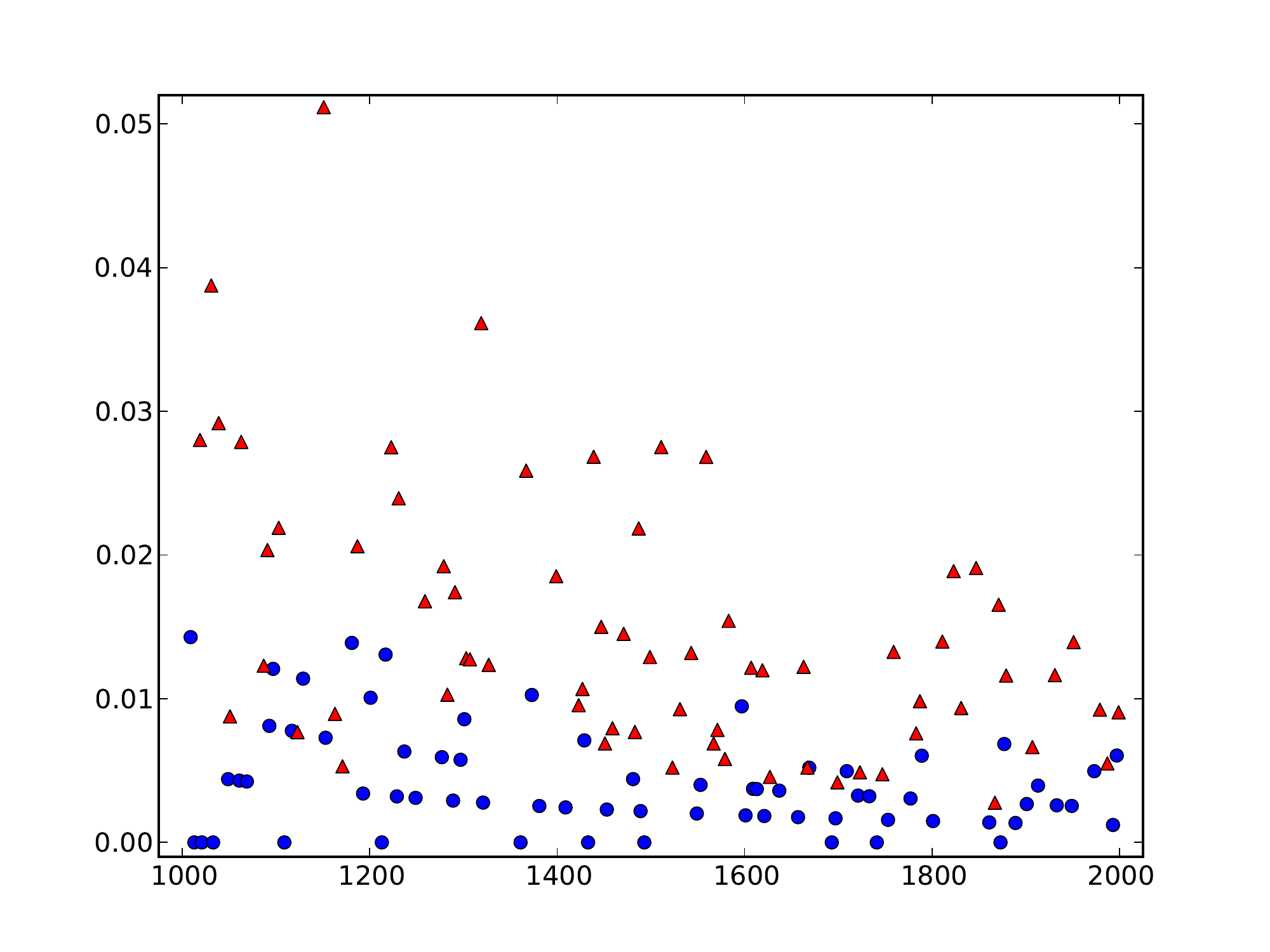}
\caption{The relative difference (as a percentage) between the actual number of 
eigensystems and the upper bound.  The three different views show the primes:
(top) $11\leq p < 2000$; (middle) $500 < p < 2000$; (bottom) $1000 < p < 2000$.
 The blue discs represent primes congruent to $1$ mod $4$, while the red 
 triangles represent primes congruent to $3$ mod $4$.}
\label{fig:reldiff}
\end{figure}

\begin{landscape}
\begin{multicols}{2}
\section*{Appendix: Table of results}
{\footnotesize
\begin{tabular}{r|r|r|l}
$p$ & number & bound & notes\\\hline
$11$ & $10$ & $10$ & \\
$13$ & $12$ & $12$ & \\
$17$ & $48$ & $48$ & \\
$19$ & $72$ & $72$ & \\
$23$ & $143$ & $154$ & \\
$29$ & $336$ & $336$ & \\
$31$ & $405$ & $420$ & \\
$37$ & $720$ & $756$ & E [32] \\
$41$ & $1080$ & $1080$ & \\
$43$ & $1260$ & $1260$ & \\
$47$ & $1656$ & $1702$ & \\
$53$ & $2496$ & $2496$ & \\
$59$ & $3393$ & $3538$ & E [44] N [16] \\
$61$ & $3900$ & $3900$ & \\
$67$ & $5148$ & $5280$ & E [58] \\
$71$ & $6195$ & $6370$ & \\
$73$ & $6840$ & $6912$ & \\
$79$ & $8736$ & $8892$ & N [38] \\
$83$ & $10373$ & $10414$ & \\
$89$ & $12848$ & $12936$ & \\
$97$ & $16896$ & $16896$ & \\
$101$ & $19100$ & $19200$ & E [68] \\
$103$ & $20196$ & $20400$ & E [24] \\
$107$ & $22737$ & $23002$ & C [26] N [28] \\
$109$ & $24300$ & $24300$ & \\
$113$ & $27104$ & $27216$ & \\
$127$ & $38934$ & $39060$ & \\
$131$ & $42510$ & $43030$ & E [22] N [40] \\
\end{tabular}

\begin{tabular}{r|r|r|l}
$p$ & number & bound & notes\\\hline
$137$ & $49368$ & $49368$ & \\
$139$ & $50991$ & $51612$ & C [20] N [36] \\
$149$ & $63788$ & $63936$ & E [130] \\
$151$ & $66075$ & $66600$ & C [52] N [60] \\
$157$ & $74256$ & $75036$ & E [62, 110] \\
$163$ & $83916$ & $84240$ & \\
$167$ & $90387$ & $90802$ & \\
$173$ & $100620$ & $101136$ & C [68] N [24] \\
$179$ & $111784$ & $112318$ & C [30] \\
$181$ & $115920$ & $116100$ & \\
$191$ & $136040$ & $136990$ & C [30] \\
$193$ & $140928$ & $141312$ & C [48] N [72] \\
$197$ & $150528$ & $150528$ & \\
$199$ & $154836$ & $155232$ & \\
$211$ & $185535$ & $185640$ & \\
$223$ & $219225$ & $219780$ & N [72] \\
$227$ & $231424$ & $232102$ & \\
$229$ & $237804$ & $238260$ & C [58, 58] \\
$233$ & $250792$ & $251256$ & E [84] \\
$239$ & $270725$ & $271558$ & \\
$241$ & $277680$ & $278400$ & C [98] \\
$251$ & $314875$ & $315250$ & \\
$257$ & $337920$ & $338688$ & E [164] N [50, 100] \\
$263$ & $362084$ & $363394$ & E [100] N [98] \\
$269$ & $388332$ & $389136$ & C [84] N [78] \\
$271$ & $396495$ & $397980$ & E [84] C [18, 40] \\
$277$ & $425040$ & $425316$ & N [92] \\
$281$ & $444360$ & $444360$ & \\
$283$ & $453033$ & $454020$ & E [20] N [72, 72] \\
$293$ & $503408$ & $504576$ & E [156] \\
\end{tabular}

\begin{tabular}{r|r|r|l}
$p$ & number & bound & notes\\\hline
$307$ & $580023$ & $581400$ & E [88] C [52] N [78] \\
$311$ & $602485$ & $604810$ & E [292] C [32, 126] \\
$313$ & $616200$ & $616512$ & N [114] \\
$317$ & $640532$ & $640848$ & \\
$331$ & $729465$ & $730620$ & C [164] N [84, 84] \\
$337$ & $771456$ & $771456$ & \\
$347$ & $842164$ & $843202$ & E [280] C [74] \\
$349$ & $857472$ & $857820$ & \\
$353$ & $886336$ & $888096$ & E [186, 300] N [76] \\
$359$ & $933127$ & $934738$ & \\
$367$ & $998448$ & $999180$ & \\
$373$ & $1049412$ & $1049412$ & \\
$379$ & $1099791$ & $1101492$ & E [100, 174] C [20] N [56] \\
$383$ & $1135686$ & $1137214$ & \\
$389$ & $1190772$ & $1191936$ & E [200] \\
$397$ & $1266804$ & $1267596$ & C [16] \\
$401$ & $1306000$ & $1306800$ & E [382] \\
$409$ & $1386792$ & $1387200$ & E [126] \\
$419$ & $1491006$ & $1492678$ & N [106] \\
$421$ & $1513260$ & $1514100$ & E [240] C [112] \\
$431$ & $1623250$ & $1625830$ & C [80] \\
$433$ & $1646352$ & $1648512$ & E [366] C [188] \\
$439$ & $1716741$ & $1718712$ & C [214] \\
$443$ & $1766232$ & $1766674$ & \\
$449$ & $1839040$ & $1839936$ & \\
$457$ & $1939824$ & $1940736$ & \\
$461$ & $1992260$ & $1992720$ & E [196] \\
$463$ & $2017323$ & $2018940$ & E [130] N [182] \\
$467$ & $2070205$ & $2072302$ & E [94, 194] \\
$479$ & $2233694$ & $2237518$ & N [236] \\
\end{tabular}

\begin{tabular}{r|r|r|l}
$p$ & number & bound & notes\\\hline
$487$ & $2351025$ & $2352240$ & \\
$491$ & $2407370$ & $2411290$ & E [292, 336, 338] C [124] N [124, 124] \\
$499$ & $2530587$ & $2531832$ & N [126] \\
$503$ & $2590320$ & $2593834$ & C [162] \\
$509$ & $2688336$ & $2688336$ & \\
$521$ & $2883400$ & $2884440$ & \\
$523$ & $2916414$ & $2917980$ & E [400] \\
$541$ & $3231360$ & $3231900$ & E [86] \\
$547$ & $3339609$ & $3341520$ & E [270, 486] \\
$557$ & $3528376$ & $3529488$ & E [222] \\
$563$ & $3644008$ & $3645694$ & \\
$569$ & $3763000$ & $3764136$ & C [86] \\
$571$ & $3803040$ & $3804180$ & \\
$577$ & $3924288$ & $3926016$ & E [52] C [54] N [36] \\
$587$ & $4132765$ & $4135402$ & E [90, 92] \\
$593$ & $4263584$ & $4264176$ & E [22] \\
$599$ & $4390516$ & $4395898$ & N [222] \\
$601$ & $4438800$ & $4440000$ & N [136] \\
$607$ & $4572876$ & $4575300$ & E [592] \\
$613$ & $4712400$ & $4713012$ & E [522] \\
$617$ & $4804184$ & $4806648$ & E [20, 174, 338] \\
$619$ & $4851300$ & $4853772$ & E [428] C [158, 216] \\
$631$ & $5140170$ & $5143320$ & E [80, 226] \\
$641$ & $5393280$ & $5393280$ & \\
$643$ & $5443839$ & $5444160$ & \\
$647$ & $5541065$ & $5547202$ & E [236, 242, 554] N [268] \\
$653$ & $5702392$ & $5703696$ & E [48] N [66] \\
$659$ & $5861135$ & $5863438$ & E [224] \\
$661$ & $5914260$ & $5916900$ & \\
$673$ & $6245568$ & $6246912$ & E [408, 502] \\
\end{tabular}

\begin{tabular}{r|r|r|l}
$p$ & number & bound & notes\\\hline
$677$ & $6357780$ & $6359808$ & E [628] \\
$683$ & $6529468$ & $6531514$ & E [32] \\
$691$ & $6762000$ & $6764760$ & E [12, 200] \\
$701$ & $7063700$ & $7064400$ & N [268] \\
$709$ & $7309392$ & $7310100$ & \\
$719$ & $7619057$ & $7625878$ & N [358] \\
$727$ & $7881456$ & $7884360$ & E [378] \\
$733$ & $8080548$ & $8082012$ & C [184] \\
$739$ & $8281836$ & $8283312$ & \\
$743$ & $8414280$ & $8419474$ & C [134] \\
$751$ & $8690625$ & $8695500$ & E [290] C [158] \\
$757$ & $8904924$ & $8906436$ & E [514] \\
$761$ & $9048560$ & $9049320$ & E [260] \\
$769$ & $9337344$ & $9338880$ & N [62] \\
$773$ & $9484792$ & $9486336$ & E [732] C [280] \\
$787$ & $10012854$ & $10013640$ & \\
$797$ & $10401332$ & $10402128$ & E [220] \\
$809$ & $10878912$ & $10881336$ & E [330, 628] \\
$811$ & $10958895$ & $10962540$ & E [544] N [140] \\
$821$ & $11373400$ & $11375040$ & E [744] \\
$823$ & $11457036$ & $11458680$ & \\
$827$ & $11624711$ & $11627602$ & E [102] \\
$829$ & $11712060$ & $11712060$ & \\
$839$ & $12133402$ & $12143458$ & E [66] N [140] \\
$853$ & $12762960$ & $12763812$ & N [68] \\
$857$ & $12943576$ & $12945288$ & C [264] \\
$859$ & $13035165$ & $13036452$ & \\
$863$ & $13215322$ & $13220494$ & \\
$877$ & $13874964$ & $13876716$ & E [868] \\
$881$ & $14066800$ & $14068560$ & E [162] \\
\end{tabular}

\begin{tabular}{r|r|r|l}
$p$ & number & bound & notes\\\hline
$883$ & $14163597$ & $14164920$ & N [222] \\
$887$ & $14352314$ & $14359402$ & E [418] \\
$907$ & $15355341$ & $15356700$ & N [228] \\
$911$ & $15553265$ & $15561910$ & C [366] \\
$919$ & $15970905$ & $15976872$ & C [120] \\
$929$ & $16504480$ & $16506336$ & E [520, 820] \\
$937$ & $16937856$ & $16937856$ & \\
$941$ & $17156880$ & $17156880$ & \\
$947$ & $17487756$ & $17488702$ & \\
$953$ & $17822392$ & $17824296$ & E [156] \\
$967$ & $18619167$ & $18624480$ & C [376, 378] \\
$971$ & $18853405$ & $18857770$ & E [166] \\
$977$ & $19210608$ & $19210608$ & \\
$983$ & $19558985$ & $19568314$ & C [144] \\
$991$ & $20046510$ & $20051460$ & C [166] \\
$997$ & $20418996$ & $20418996$ & \\
$1009$ & $21164976$ & $21168000$ & C [126] \\
$1013$ & $21422016$ & $21422016$ & \\
$1019$ & $21800470$ & $21806578$ & C [356] \\
$1021$ & $21935100$ & $21935100$ & \\
$1031$ & $22580175$ & $22588930$ & \\
$1033$ & $22720512$ & $22720512$ & \\
$1039$ & $23113665$ & $23120412$ & \\
$1049$ & $23795888$ & $23796936$ & N [426] \\
$1051$ & $23931600$ & $23933700$ & N [368] \\
$1061$ & $24624860$ & $24625920$ & E [474] \\
$1063$ & $24758937$ & $24765840$ & N [352] \\
$1069$ & $25187712$ & $25188780$ & N [280] \\
$1087$ & $26484282$ & $26487540$ & N [52] \\
$1091$ & $26776940$ & $26782390$ & E [888] \\
\end{tabular}

\begin{tabular}{r|r|r|l}
$p$ & number & bound & notes\\\hline
$1093$ & $26927628$ & $26929812$ & C [164, 460] \\
$1097$ & $27224640$ & $27227928$ & C [324, 408] \\
$1103$ & $27672873$ & $27678934$ & \\
$1109$ & $28134336$ & $28134336$ & \\
$1117$ & $28747044$ & $28749276$ & E [794] N [476] \\
$1123$ & $29214636$ & $29216880$ & N [152] \\
$1129$ & $29685576$ & $29688960$ & E [348] N [192] \\
$1151$ & $31449050$ & $31465150$ & E [534, 784, 968] \\
$1153$ & $31627008$ & $31629312$ & E [802] \\
$1163$ & $32459889$ & $32462794$ & \\
$1171$ & $33137325$ & $33139080$ & \\
$1181$ & $33993440$ & $33998160$ & C [360] N [182] \\
$1187$ & $34513786$ & $34520902$ & N [114, 254, 298] \\
$1193$ & $35047184$ & $35048376$ & E [262] \\
$1201$ & $35756400$ & $35760000$ & E [676] C [460] \\
$1213$ & $36846012$ & $36846012$ & \\
$1217$ & $37208384$ & $37213248$ & E [784, 866, 1118] \\
$1223$ & $37757967$ & $37768354$ & \\
$1229$ & $38327108$ & $38328336$ & E [784] \\
$1231$ & $38506995$ & $38516220$ & N [100] \\
$1237$ & $39081084$ & $39083556$ & E [874] \\
$1249$ & $40234272$ & $40235520$ & N [224] \\
$1259$ & $41206419$ & $41213338$ & N [316] \\
$1277$ & $43008856$ & $43011408$ & C [540] N [532] \\
$1279$ & $43205985$ & $43214292$ & E [518] \\
$1283$ & $43618127$ & $43622614$ & E [510] \\
$1289$ & $44237648$ & $44238936$ & \\
$1291$ & $44437920$ & $44445660$ & E [206, 824] N [324] \\
$1297$ & $45067104$ & $45069696$ & E [202, 220] \\
$1301$ & $45485700$ & $45489600$ & E [176] \\
\end{tabular}

\begin{tabular}{r|r|r|l}
$p$ & number & bound & notes\\\hline
$1303$ & $45694341$ & $45700200$ & C [410] \\
$1307$ & $46118125$ & $46124002$ & E [382, 852] \\
$1319$ & $47392644$ & $47409778$ & E [304] \\
$1321$ & $47624280$ & $47625600$ & C [168] \\
$1327$ & $48273693$ & $48279660$ & E [466] \\
$1361$ & $52097520$ & $52097520$ & \\
$1367$ & $52778142$ & $52791802$ & E [234] \\
$1373$ & $53486048$ & $53491536$ & C [344] N [444, 520] \\
$1381$ & $54432720$ & $54434100$ & E [266] \\
$1399$ & $56586147$ & $56596632$ & \\
$1409$ & $57820928$ & $57822336$ & E [358] \\
$1423$ & $59561892$ & $59567580$ & \\
$1427$ & $60066685$ & $60073102$ & N [358] \\
$1429$ & $60321576$ & $60325860$ & E [996] C [94] \\
$1433$ & $60835656$ & $60835656$ & \\
$1439$ & $61588821$ & $61605358$ & E [574] N [674] \\
$1447$ & $62631321$ & $62640720$ & \\
$1451$ & $63159100$ & $63163450$ & \\
$1453$ & $63423360$ & $63424812$ & N [702] \\
$1459$ & $64211049$ & $64216152$ & \\
$1471$ & $65808225$ & $65817780$ & \\
$1481$ & $67169800$ & $67172760$ & N [530] \\
$1483$ & $67440633$ & $67445820$ & E [224] N [694] \\
$1487$ & $67980042$ & $67994902$ & \\
$1489$ & $68267952$ & $68269440$ & \\
$1493$ & $68822976$ & $68822976$ & \\
$1499$ & $69649510$ & $69658498$ & E [94] \\
$1511$ & $71329380$ & $71349010$ & C [498] \\
$1523$ & $73062849$ & $73066654$ & E [1310] \\
$1531$ & $74219535$ & $74226420$ & N [252] \\
\end{tabular}

\begin{tabular}{r|r|r|l}
$p$ & number & bound & notes\\\hline
$1543$ & $75979737$ & $75989760$ & C [732] \\
$1549$ & $76879872$ & $76881420$ & C [110] \\
$1553$ & $77477392$ & $77480496$ & N [620] \\
$1559$ & $78363505$ & $78384538$ & E [862] \\
$1567$ & $79594299$ & $79599780$ & \\
$1571$ & $80206590$ & $80212870$ & \\
$1579$ & $81442158$ & $81446892$ & N [396] \\
$1583$ & $82056758$ & $82069414$ & \\
$1597$ & $84262416$ & $84270396$ & E [842] C [168, 196, 398] \\
$1601$ & $84905600$ & $84907200$ & \\
$1607$ & $85857563$ & $85868002$ & \\
$1609$ & $86185584$ & $86188800$ & E [1356] \\
$1613$ & $86831992$ & $86835216$ & E [172] \\
$1619$ & $87799961$ & $87810478$ & E [560] N [406] \\
$1621$ & $88134480$ & $88136100$ & E [980] \\
$1627$ & $89116995$ & $89121060$ & N [644] \\
$1637$ & $90775096$ & $90778368$ & E [718] N [714] \\
$1657$ & $94151880$ & $94153536$ & C [176] \\
$1663$ & $95171106$ & $95182740$ & E [270, 1508] C [396] \\
$1667$ & $95868304$ & $95873302$ & \\
$1669$ & $96213576$ & $96218580$ & E [388, 1086] C [652] \\
$1693$ & $100438812$ & $100438812$ & \\
$1697$ & $101152832$ & $101154528$ & C [432] \\
$1699$ & $101508987$ & $101513232$ & \\
$1709$ & $103315212$ & $103320336$ & C [72, 514] \\
$1721$ & $105513400$ & $105516840$ & E [30] \\
$1723$ & $105880614$ & $105885780$ & N [488] \\
$1733$ & $107740792$ & $107744256$ & E [810, 942] \\
$1741$ & $109245900$ & $109245900$ & \\
$1747$ & $110376882$ & $110382120$ & \\
\end{tabular}

\begin{tabular}{r|r|r|l}
$p$ & number & bound & notes\\\hline
$1753$ & $111523560$ & $111525312$ & E [712] \\
$1759$ & $112662309$ & $112677252$ & E [1520] \\
$1777$ & $116175264$ & $116178816$ & E [1192] \\
$1783$ & $117353610$ & $117362520$ & C [762] \\
$1787$ & $118144793$ & $118156402$ & E [1606] N [358, 498] \\
$1789$ & $118546188$ & $118553340$ & E [848, 1442] \\
$1801$ & $120958200$ & $120960000$ & C [728] \\
$1811$ & $122974115$ & $122991310$ & E [550, 698, 1520] N [824] \\
$1823$ & $125433768$ & $125457454$ & \\
$1831$ & $127107225$ & $127119120$ & E [1274] \\
$1847$ & $130463281$ & $130488202$ & E [954, 1016, 1558] \\
$1861$ & $133481040$ & $133482900$ & \\
$1867$ & $134777448$ & $134781180$ & \\
$1871$ & $135629230$ & $135651670$ & E [1794] \\
$1873$ & $136086912$ & $136086912$ & \\
$1877$ & $136953628$ & $136963008$ & E [1026] C [516] N [278] \\
$1879$ & $137386029$ & $137401992$ & E [1260] \\
$1889$ & $139610048$ & $139611936$ & E [242] \\
$1901$ & $142291000$ & $142294800$ & E [1722] C [476] \\
$1907$ & $143639972$ & $143649502$ & C [368] \\
$1913$ & $145006080$ & $145011816$ & C [702] N [872] \\
$1931$ & $149134960$ & $149152330$ & C [296] N [456, 484, 484] \\
$1933$ & $149612148$ & $149616012$ & E [1058, 1320] \\
$1949$ & $153366040$ & $153369936$ & C [44, 170] \\
$1951$ & $153821850$ & $153843300$ & E [1656] \\
$1973$ & $159108848$ & $159116736$ & C [900] N [70, 248] \\
$1979$ & $160561183$ & $160576018$ & E [148] \\
$1987$ & $162525303$ & $162534240$ & E [510] C [770] \\
$1993$ & $164011320$ & $164013312$ & E [912] \\
$1997$ & $164995348$ & $165005328$ & E [772, 1888] N [562] \\
$1999$ & $165487347$ & $165502332$ & \\
\end{tabular}
}\end{multicols}
\end{landscape}

\bibliographystyle{halpha}
\bibliography{eigenlist}
\end{document}